\documentclass[aps,prd,groupedaddress, showkeys]{revtex4-2}

\usepackage{amsmath}
\usepackage{bbding} 
\usepackage{bbold}
\usepackage{dsfont}
\usepackage{amssymb}
\usepackage{amsthm}
\usepackage{amsfonts}
\usepackage{xcolor}
\usepackage{setspace}
\usepackage{caption}
\usepackage{subcaption}
\usepackage{lipsum}
\usepackage{tcolorbox} 
\usepackage{mathtools} 
\usepackage{tgpagella}
\usepackage{enumerate}
\usepackage{fullpage} 
\usepackage[english]{babel}
\usepackage{blindtext}

\newcommand{\rfw}[1]{Eq.\eqref{eq:#1}}
\newcommand{\reffig}[1]{Fig. \ref{fig:#1}}

\newcommand{\mm}[0]{\nonumber \\}

\newcommand{\tbf}[1]{\textbf{#1}}

\newcommand{\sgn}[0]{\text{sgn}\,}

\DeclareMathOperator{\diag}{diag}
\DeclareMathOperator{\perm}{perm}

\newtheorem{theorem}{Theorem}
\newtheorem{corollary}{Corollary}[theorem]

\allowdisplaybreaks

\usepackage{hyperref} 

\begin{document}


\title{Permanents through probability distributions}


\author{Mobolaji Williams}
\email[]{mwilliams@jellyfish.co}
\affiliation{Jellyfish, Boston, MA 02111 }

\date{June 22, 2021}
\begin{abstract}

We show that the permanent of a matrix can be written as the expectation value of a function of random variables each with zero mean and unit variance. This result is used to show that Glynn's theorem and a simplified MacMahon theorem extend from a common probabilistic interpretation of the permanent. Combining the methods in these two proofs, we prove a new result that relates the permanent of a matrix to the expectation value of a product of hyperbolic trigonometric functions, or, equivalently, the partition function of a spin system. We conclude by discussing how the main theorem can be generalized and how the techniques used to prove it can be applied to more general problems in combinatorics.
\end{abstract}


\maketitle

\section{Introduction \label{sec:intro}}

The permanent is a seemingly odd function in linear algebra. For a matrix $A$ with row $i$ and column $j$ element $a_{i, j}$, it is defined as 
\begin{equation}
\perm(A) = \sum_{\sigma \in S_{n}} \prod_{i=1}^n a_{i, \sigma(i)},
\label{eq:perm_def}
\end{equation} 
where $\sigma$ is an element of the symmetric group $S_n$. The permanent looks similar to the determinant, but it is missing the ``sign of the permutation" factor $\sgn(\sigma)$ which enters into a determinant and provides the latter with much of its computational simplicity. For example, while algorithms exist to compute the determinant in polynomial time, calculations (admittedly naive ones) of the permanent require $O(n!)$ time. Consequently, there is a literature around permanents that seeks alternative ways of expressing the permanent for the goal of finding either the new ways to compute it efficiently or new ways to understand it. 

Glynn's theorem \cite{Glynn_2010} falls in the ``computational efficiency" category and provides a way to compute the permanent in $O(2^{n-1}n^2)$ time. The MacMahon Master theorem \cite{Percus_1971} falls in the ``new ways of understanding" category and relates permanents to their linear algebra cousin determinants. And though both Glynn's theorem and the MacMahon theorem are related to permanents, they are derived from quite different starting points. 

In this work, we use a probabilistic definition of the permanent to show that the two results extend from the same foundation, and we then use this definition to derive a new result combining aspects of both theorems. 

We start with the main theorem we seek to prove: 

\begin{theorem} Let $p_X: \Omega_X \to \mathbb{R}$ be a probability distribution defined over the domain $\Omega_X$ with zero mean and unit variance. Let $A$ be an $n\times n$ matrix with elements $a_{i, j}$. Then the permanent of $A$ is 
\begin{equation}
\perm(A) = \int_{\Omega^n_{X}} d^n\tbf{x}\, \prod_{i=1}^n p_{X}(x_i) \, x_{i} \sum_{j=1}^n a_{i, j} x_j,
\label{eq:fund_thm}
\end{equation}
where $\Omega^n_X = \Omega_X  \otimes \cdots \otimes \Omega_X$ is the $n$-factor product over the single-variable domain of integration. In condensed notation, we can write \rfw{fund_thm} as the expectation value
\begin{equation}
\perm(A) = \left\langle \prod_{i=1}^nx_{i} \sum_{j=1}^n a_{i, j} x_j \right\rangle_{x_i \sim p_X},
\end{equation}
where the average is over $\{x_i\}$, a set of independent identically distributed random variables each of which is drawn from $p_X$. 
\label{thm:fund_thm}
\end{theorem}

\begin{proof} Say we have a single-variable function $p_X$, defined over the domain $\Omega_{X} \subset \mathbb{R}$, which has the following two properties:
\begin{equation}
\int_{\Omega_X} dx\, p_X(x) x = 0, \qquad \int_{\Omega_X} dx\, p_X(x) x^2 = 1. 
\label{eq:p_props}
\end{equation}
Interpreting $x$ as a random variable, the conditions in \rfw{p_props} are equivalent to requiring $p_X$ to be a probability distribution with zero mean and unit variance. We use this interpretation for simplicity in defining the properties of $p_X$. Two such functions are $p_X(x) = \frac{1}{2}\delta(x-1) + \frac{1}{2} \delta(x+1)$ and $p_X(x) = e^{-x^2/2}/\sqrt{2\pi}$. 

Next we define a quantity called the \textit{permutation tensor}:
\begin{equation}
\Delta_{\tbf{m}} =\begin{dcases} 1 & \text{if $\, \tbf{m}$ is a permutation of $(1, 2, \ldots, n)$} \\ 0 & \text{otherwise} \end{dcases},
\label{eq:perm_tensor1}
\end{equation}
where $\tbf{m} = (m_1, m_2, \ldots, m_n)$.
That is, $\Delta_{\tbf{m}} $ is a function of $N$ integers $\tbf{m} = (m_1, m_2, \ldots, m_n)$, and it is non-zero and equal to 1 if and only if the integers are a permutation of the numbers from $1$ to $n$. The permanent of $A$ \rfw{perm_def} can be written in terms of the permutation tensor as 
\begin{equation}
\perm(A) = \sum_{\{m_j\}} \Delta_{\tbf{m}} \prod_{j=1}^{n}a_{j,m_{j}},
\label{eq:perm_def2}
\end{equation}
where the summation is defined as 
\begin{equation}
\sum_{\{m_j\}} \equiv \prod_{i=1}^{n} \sum_{m_i=1}^n.
\end{equation}
There are many possible forms for $\Delta_{\tbf{m}}$, but we will use \rfw{p_props} to write is as the multiple integral expression 
\begin{equation}
\Delta_{\tbf{m}}  = \int_{\Omega^n_X} d^n \tbf{x}\, \prod_{i=1}^{n}p_X(x_i) x_{i} x_{m_i},
\label{eq:perm_tensor2}
\end{equation}
where $d^n\tbf{x} = dx_1dx_2\cdots dx_n$. Substituting \rfw{perm_tensor2} into \rfw{perm_def2}, we obtain 
\begin{align}
\perm(A)&  = \sum_{\{m_j\}}\int_{\Omega^n_X} d^n \tbf{x}\, \left[\prod_{i=1}^{n}p_X(x_i) x_{i} x_{m_i}\right]\left[\prod_{j=1}^{n}a_{j,m_{j}} \right]\mm
& = \left[ \prod_{i=1}^{n} \sum_{m_i=1}^n\right]\int_{\Omega^n_X} d^n \tbf{x}\, \prod_{i=1}^{n}p_X(x_i) x_{i}a_{i, m_{i}} x_{m_i} \mm
& = \int_{\Omega^n_X} d^n \tbf{x}\, \prod_{i=1}^{n}p_X(x_i) x_{i}\sum_{m_i=1}^na_{i, m_{i}} x_{m_i},
\end{align}
and the theorem is proved. 

\end{proof}

As a sanity check, we can use \rfw{fund_thm} to affirm a basic result in permanent calculations. We know that if $A = \mathds{1}$, i.e., $A$ is a matrix entirely composed of $1$s, then the permanent is $n!$. Checking this result, we have 
\begin{align}
\perm(\mathds{1}) 
& = \int_{\Omega^n_{X}} d^n\tbf{x}\, \prod_{i=1}^n p_{X}(x_i) \, x_{i} \sum_{j=1}^n  x_j\mm
 & = \int_{\Omega^n_{X}} d^n\tbf{x}\, \left[\prod_{i=1}^n p_{X}(x_i)\, x_{i} \right] \Big( \sum_{j=1}^n  x_j\Big)^{n}\mm
  & =\sum_{k_1 + \cdots + k_n = n}  n!\prod_{i=1}^{n}\frac{1}{k_i!}\int_{\Omega_X} dx_i\, p_X(x_i) x_{i}^{k_i +1}
\label{eq:fund_thm}
\end{align}
This final result is only nonzero for the integrands where $x_i$ is quadratic for each $i$. Given that $k_1 + \cdots +k_n = n$, we can only have quadratic $x_i$ for all $i$ if $k_i = 1$ for all $i$. Thus, we have 
\begin{align}
\perm(\mathds{1})  &= \sum_{k_1 + \cdots + k_n =n} n! \prod_{i=1}^{n}\frac{1}{k_i!} \delta(1, k_i) = n!,
\end{align}
as expected. 

Identifying $p_X$ as a probability distribution is primarily for convenience to allow us to reference the integrals over the linear and quadratic variable as a mean and variance, respectively. The probability-based interpretation of these conditions is valid, but as we explain in the discussion, there is a more general way  to write this condition.

But first we explore how this result relates to well-known theorems involving permanents. First, with Theorem \ref{thm:fund_thm}, it becomes straightforward to prove Glynn's theorem for computing permanents\cite{Glynn_2010}. 

\begin{corollary}[Glynn's Theorem] Let $A = (a_{i, j})$ be an $n \times n$ matrix. Then we have
\begin{equation}
\perm(A) = \frac{1}{2^{n-1}}\left[\sum_{\{s_i\}; s_1 = +1} \left(\prod_{k=1}^n s_k\right) \prod_{j=1}^n \sum_{i=1}^n  a_{i, j}s_j\right],
\label{eq:glynn_res}
\end{equation}
where the outer sum is over all $2^{n-1}$ vectors $s = (s_1, \ldots, s_n)\in \{\pm 1\}^n$ with $s_1 = 1$. 
\label{corr:glynn}
\end{corollary}

\begin{proof}
We start with Theorem \ref{thm:fund_thm}. We use the probability density 
\begin{equation}
p_X(s) = \frac{1}{2} \delta(s+1)  + \frac{1}{2} \delta(s-1),
\label{eq:dirac}
\end{equation}
where $\delta(x)$ is the Dirac delta function (pg. 160 \cite{hassani2013mathematical}). \rfw{dirac} expresses $X$ as a Bernoulli distribution with values $X=\pm1$ having equal probabilities. It is straightforward to check that \rfw{dirac} yields zero mean and unit variance.  Inserting \rfw{dirac} into \rfw{fund_thm}, and noting that 
\begin{equation}
\int_{\Omega_X} dx\,p_X(x) f(x) = \frac{1}{2} \sum_{s \in \{ \pm 1\}} f(s),
\end{equation}
we find 
\begin{equation}
\perm(A) = \frac{1}{2^{n}}\sum_{\{s_i\}} \prod_{i=1}^n\, s_{i} \sum_{j=1}^n a_{i, j} s_j, 
\label{eq:glynn_prelim}
\end{equation}
where 
\begin{equation}
 \sum_{\{s_i\}} = \prod_{i=1}^n \sum_{s_i\in \{-1,+1\}}
\end{equation}
\rfw{glynn_prelim} is a preliminary result for the permanent expressed as an an expectation value over $n$ Bernoulli random variables, and it looks very close to the standard result \rfw{glynn_res}.

To derive the standard result, we first recognize that \rfw{glynn_prelim} is unchanged if we take $s_i$ to $-s_i$ for all $i$. We can see this most clearly by noting that the factor $ s_{i} \sum_{j=1}^n a_{i, j} s_{j}$ does not change when we take $s_i, s_j \to - s_i , -s_j$. Thus for each term in the summation, there is another term that is identical to it for which all of its $s_i$ values are switched from positive to negative or vice versa. We can separate these two sets of terms by isolating a particular $s_i$ and setting this $s_i = +1$ for one set and $s_i = -1$ for the other. Without loss of generality we take this isolated $s_i$ to be at $i=1$. We can then write \rfw{glynn_prelim} as
\begin{align}
\perm(A) & = \frac{1}{2^n}\sum_{\{s_i\}} \prod_{i=1}^n s_{i} \sum_{j=1}^n a_{i, j} s_{j}\mm
& =   \frac{1}{2^{n}}\sum_{\{s_i\}; s_1 = +1}\prod_{i=1}^n s_{i} \sum_{j=1}^n a_{i, j} s_{j} +  \frac{1}{2^{n-1}}\sum_{\{s_i\}; s_1 = -1}\prod_{i=1}^n s_{i} \sum_{j=1}^n a_{i, j} s_{j},\mm
& =  \frac{1}{2^{n-1}}\sum_{\{s_i\}; s_1 = +1} \prod_{i=1}^n s_{i} \sum_{j=1}^n a_{i, j} s_{j}.
\label{eq:spin_perm_tensor2}
\end{align}
In the second equality we split the summation into the aforementioned two sets of terms, and in the third equality we set these two sets equal due to the noted $\{s_i\} \to -\{s_i\}$ symmetry. Thus Glynn's theorm is proved. 
\end{proof}

The final result in \rfw{spin_perm_tensor2} involves a summation over $2^{n-1}$ terms each of which involves an $n$-factor product over an $n$-term sum. Thus the computation time for the permanent under \rfw{spin_perm_tensor2} is $O(2^{n-1}n^2)$ which is much better than the $O(n!)$ computation time obtained from a brute force calculation according to the definition \rfw{perm_def}. 

The essential advantage of Theorem \ref{thm:fund_thm} is that it is written in terms of the arbitrary probability distribution $p_X$ and thus we can make a convenient choice for this distribution when we seek to prove any theorem. We make such a choice in our proof of a simplified version of the MacMahon Master theorem. 

\begin{corollary}[Simplified MacMahon Master Theorem]  Let $A = (a_{i, j})$ be an $n \times n$ matrix. Let $X$ be the diagonal matrix $X = \diag(x_1, x_2, \ldots, x_n)$.  Then we have
\begin{equation}
\operatorname{perm}(A) = \text{coefficient of $x_1x_2\cdots x_n$ }\;\text{ term in } \; \frac{1}{\det(I - X A)},
\end{equation}
where $I$ is the $n\times n$ identity matrix. 
\label{corr:mahon}
\end{corollary}

\begin{proof}
We again start with \rfw{fund_thm}. We choose our probability density to be a Gaussian $p_X(x) = e^{-x^2/2}/\sqrt{2\pi}$. \rfw{fund_thm} then becomes 
\begin{equation}
\perm(A) = \frac{1}{(2\pi)^{n/2}}\int_{\mathbb{R}^n} d^n\tbf{x}\,e^{-\tbf{x}^2/2} \prod_{i=1}^n \, x_{i} \sum_{j=1}^n a_{i, j} x_j,
\label{eq:fund_thm_gauss}
\end{equation}
where $\tbf{x}^2 = x_1^2 + \cdots + x_n^2$.
Using the contour integral identity 
\begin{equation}
\prod_{k=1}^{n}\alpha_k = \frac{1}{(2\pi i)^n} \oint\left[ \prod_{k=1}^{n} \frac{dz_k}{z_k^2}\,\right] e^{\sum_{k=1}^n\alpha_k z_k},
\label{eq:cont_exp_id}
\end{equation}
we write the product in \rfw{fund_thm_gauss} as 
\begin{equation}
 \prod_{i=1}^n \, x_{i} \sum_{j=1}^n a_{i, j} x_j = \frac{1}{(2\pi i)^n} \oint \left[\prod_{i=1}^n\frac{dz_i}{z_i^2} \right]\exp \left(\sum_{i=1}^nz_i x_i \sum_{j=1}^n a_{i, j} x_{j} \right).
\end{equation} 
Inserting this form of the product into \rfw{fund_thm_gauss}, we have 
\begin{align}
\perm&(A) \mm
&  = \frac{1}{(2\pi i)^n} \oint \left[\prod_{i=1}^n\frac{dz_i}{z_i^2} \right]\frac{1}{(2\pi)^{n/2}}\int_{\mathbb{R}^n} d^n\tbf{x}\, \exp\left( - \frac{1}{2}\sum_{i=1}^nx_i^2 + \sum_{i=1}^nz_i x_i \sum_{j=1}^n a_{i, j} x_{j}\right)\mm
& =\frac{1}{(2\pi i)^n} \oint \left[\prod_{i=1}^n\frac{dz_i}{z_i^2} \right] \frac{1}{(2\pi)^{n/2}}\int_{\mathbb{R}^n} d^n\tbf{x}\, \exp\left( - \frac{1}{2}\sum_{i, j=1}^Nx_i\left(\delta_{i, j} - 2 z_ia_{i, j}\right)x_{j} \right)\mm
& = \frac{1}{(2\pi i)^n} \oint \left[\prod_{i=1}^n\frac{dz_i}{z_i^2} \right] \frac{1}{\det{(I - 2Z A)^{1/2}}}
\label{eq:perm_contour}
\end{align}
From Cauchy's integral theorem \cite{ahlfors1953complex}, we know that the contour integral $\frac{1}{2\pi i} \oint \frac{dz}{z^2} p(z)$ picks out the linear-in-$x$ term of $p(x)$. Thus we see that the permanent of $A$ is the coefficient of the $x_1x_2\cdots x_n$ term in ${1}/{\det(I - 2X A)^{1/2}}$. Since it is only the linear-in-$z$ term within the contour integral that contributes to the permanent, we can add a quadratic-in-$Z$ term to the argument of the determinant in \rfw{perm_contour} with no change to the final result: 
\begin{equation}
\perm(A) = \frac{1}{(2\pi i)^n} \oint \left[\prod_{i=1}^n\frac{dz_i}{z_i^2} \right] \frac{1}{\det{(I - 2Z A + ZAZA)^{1/2}}}.
\end{equation}
With the fact that $\det(Q^2)^{1/2} = (\det(Q) \det(Q))^{1/2} = \det(Q)$ for an arbitrary matrix $Q$, we then have 
\begin{equation}
\perm(A) = \frac{1}{(2\pi i)^n} \oint \left[\prod_{i=1}^n\frac{dz_i}{z_i^2} \right] \frac{1}{\det{(I - Z A)}}.
\label{eq:perm_contour2}
\end{equation}

Using the same reasoning as that presented below \rfw{perm_contour}, we see that \rfw{perm_contour2} allows us to conclude that the permanent of $A$ is the coefficient of the $x_1x_2\cdots x_n$ term in ${1}/{\det(I - X A)}$.  

\end{proof}

\begin{figure}[h]
\centering
\includegraphics[width=0.45\linewidth]{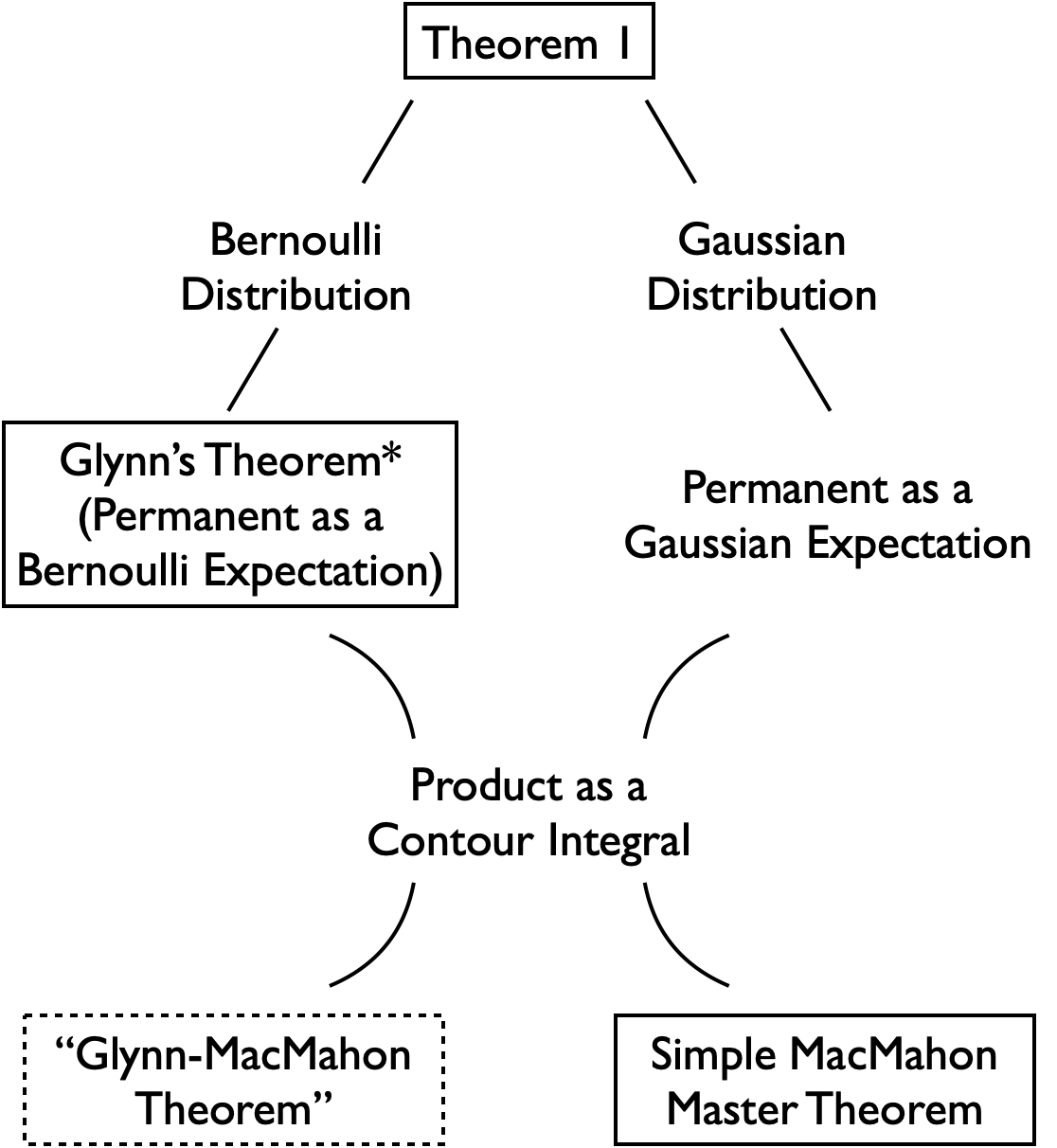}
	\caption{Theorem map: Map showing how Theorem \ref{thm:fund_thm} is related to Glynn's theorem and the simplified MacMahon Master theorem. The ``*" next to Glynn's theorem represents the fact that we are referencing \rfw{glynn_prelim}. The ``Permanent as a Gaussian Expectation" refers to \rfw{fund_thm_gauss}. The ``Product as a Contour Integral" result refers to \rfw{cont_exp_id}. Although the Gaussian integral analog of Glynn's theorem, \rfw{fund_thm_gauss}, is known (e.g., pg 27 of \cite{Percus_1971}), the Bernoulli distribution analog of the simplified MacMahon Master theorem is not part of the literature on permanents. Motivated by this gap we derive the necessary analog, terming it the Glynn-MacMahon theorem given its placement on the map. }
\label{fig:theorem_map}
\end{figure}

With Glynn's theorem and a simplified MacMahon Master theorem proved from Theorem \ref{thm:fund_thm}, we can create a map (\reffig{theorem_map}) relating the former two theorems to their common starting point. From the map, we see that starting from Theorem \ref{thm:fund_thm}, we followed paths that ultimately led us to Glynn's theorem and the MacMahon Master theorem by choosing, respectively, a Bernoulli distribution and a Gaussian distribution as the particular probability distributions from which to sample $x_i$ in \rfw{fund_thm}. However, the map suggests there is a gap in this collection of results. While Glynn's theorem is the Bernoulli distribution analog of \rfw{fund_thm_gauss}, there appears to be no Bernoulli distribution analog to the MacMahon Master theorem. 

In the next theorem, we fill in this gap by proving what we term the Glynn-MacMahon Master theorem given the result's placement in \reffig{theorem_map}. The theorem shows that the permanent of a matrix can be related to the expectation value of a product over hyperbolic trigonometric functions or, equivalently, the partition function of a spin system.

\begin{corollary}[Glynn-MacMahon Theorem] Let $A = (a_{i, j})$ be an $n \times n$ matrix with non-zero determinant, let $X$ be the diagonal matrix $X = \diag(x_1, x_2, \ldots, x_n)$, and let $(\varphi_1, \ldots, \varphi_n)$ to be a vector of $n$ normally distributed random variables with mean $0$ for all components and covariance matrix $XA$. Then we have 
\begin{equation}
\operatorname{perm}(A) = \text{coefficient of $x_1x_2\cdots x_n$} \;\text{ term in }\; \left\langle\prod_{i=1}^n 2 \cosh \varphi_i \right\rangle_{\varphi_i \sim {\cal N}(0 , XA)}.
\end{equation}
\label{corr:glynn_macmahon}
\end{corollary}

\begin{proof}
From an intermediate result in the proof of Glynn's Theorem, we can write the permanent of $A$ as an expectation value over Bernoulli random variables:
\begin{equation}
\perm(A) = \sum_{\{s_i\}} \prod_{i=1}^n\, \frac{1}{2}\sum_{j=1}^n s_{i}a_{i, j} s_j\,\qquad \text{where} \quad
 \sum_{\{s_i\}} = \prod_{i=1}^n \sum_{s_i\in \{-1,+1\}}.
\label{eq:glynn_prelim2}
\end{equation}
Using \rfw{cont_exp_id}, we then obtain 
\begin{align}
\perm(A) & = \frac{1}{(2\pi i)^{n}} \oint \left[ \prod_{k=1}^n \frac{dz_k}{z_k^2}\right] \sum_{\{s_i\}} \exp\left[ \frac{1}{2} \sum_{i, j=1}^n s_{i}z_ia_{i, j} s_j\right]
\label{eq:perm_spin_prelim}.
\end{align}

Next, we take $z_{i} a_{i, j}$ to be the component in the $i$th row and $j$th column of $ZA$ where $Z = \diag(z_1, \ldots, z_n)$. For notational simplicity, we define $W \equiv ZA$ where $Z = \diag(z_1, \ldots, z_n)$.  Then, using the evaluation of Gaussian integrals for multivariate functions, we have
\begin{align}
\exp\left[ \frac{1}{2} \sum_{i, j=1}^n s_{i}W_{i, j} s_j\right] & = \left(\frac{\det W^{-1}}{(2\pi )^n}\right)^{1/2} \mm
& \times \int d^{n} \boldsymbol{\varphi} \, \exp \left[ - \frac{1}{2} \sum_{i, j=1}^n \varphi_i (W^{-1})_{i,j} \varphi_{j} + \sum_{i=1}^{n} \varphi_i s_{i} \right].
\label{eq:gauss_spin}
\end{align}
From our definition of $W$, we can show that $\det W = z_1 \cdots z_n \det A$, and therefore this result is valid only if $\det A  \neq 0$. 
Performing the summation over spins, and noting that the prefactor in \rfw{gauss_spin} is a normalization constant, we have 
\begin{equation}
\sum_{\{s_i\}} \exp\left[ \frac{1}{2} \sum_{i, j=1}^n s_{i}W_{i, j} s_j\right] = \left\langle\prod_{i=1}^n 2 \cosh \varphi_i\right\rangle_{\varphi_i \sim {\cal N}(0, W)},
\end{equation}
where $\langle {\cal O}(\{\varphi_i\})\rangle_{\varphi_i \sim {\cal N}(0, W)}$ indicates we are taking the expectation value of the function ${\cal O}$ over the normally distributed values $(\varphi_1, \ldots, \varphi_n)$ with mean $0$ and covariance matrix $W$. From \rfw{perm_spin_prelim}, and the fact that the contour integrals pick out the linear parts of their corresponding variable, the theorem is proved. 

\end{proof}

An intermediate result in the above proof shows how the permanent is connected to a canonical system in physics. In statistical physics, if we have a set of $n$ spin variables denoted $\{s_i\}$, where each spin $i$ in this set can have the value $s_i = +1$ or $s_i = -1$, and these spins interact according to the matrix $J_{i,j}$, then the thermal partition function of the system is given by  \cite{Nishimori_2001}
\begin{equation}
Z^{\text{spin}}_{n}(\{\beta J\}) = \sum_{\{s_i\}} \exp\left[- \frac{\beta}{2} \sum_{i, j=1}^n s_{i}J_{i, j} s_j\right].
\end{equation}
where $\beta$ is inverse temperature. By comparison with \rfw{perm_spin_prelim}, we find that the permanent is equal to the $x_1 \cdots x_n$ term of the $n$-spin partition function with interaction matrix $x_{i} a_{i, j}/\beta$, namely:
\begin{equation}
\operatorname{perm}(A) = \text{coefficient of $x_1x_2\cdots x_n$} \;\text{ term in }\; Z^{\text{spin}}_{n}(\{ -XA \}).
\end{equation}

\section{Discussion}
In this work we proved a general theorem for permanents, and used the result to prove Glynn's theorem and a simplified MacMahon Master theorem. Combining aspects of each proof, we proved a new result relating permanents to the expectation value of a product of hyperbolic trigonometric functions. The fact that we can obtain such disparate results from the same theorem suggests that the starting theorem can serve as a foundational definition for permanents. 

In general, although it is convenient to do so, we do not need to interpret the $p_X$ that appears in Theorem \ref{thm:fund_thm} as a probability density. Moreover, it is possible to write the permanent as any multiple integral where the single-variable integrals have two key properties: Their results are zero when their integrands are linear in a function and unity when their integrands are quadratic in that function. 
That is, for functions $f, g: \Omega_X \to \mathbb{R}$ and the set of conditions 
\begin{equation}
\int_{\Omega_X} dx\, g(x) f(x) = 0 \qquad \int_{\Omega_X} dx\, g(x) f(x)^2 = 1,
\label{eq:p_props2}
\end{equation}
we could have just have easily proved  
\begin{equation}
\perm(A) = \int_{\Omega^n_{X}} d^n\tbf{x}\, \prod_{i=1}^n g(x_i) \, f(x_{i}) \sum_{j=1}^n a_{i, j} f(x_j),
\label{eq:fund_thm2}
\end{equation}
using an approach identical to that for Theorem \ref{thm:fund_thm}. This more general case reveals that there is not a coherent probabilistic interpretation for all such expressions. For example, the permanent of $A$, written as a multiple integral over a product of sine functions, can be expressed as
\begin{equation}
\perm(A) = \frac{1}{\pi^n}\int^{2\pi}_{0} d\theta_1 \ldots d\theta_N\, \prod_{i=1}^N\sin(\theta_i) \sum_{j=1}^n a_{i, j} \sin(  \theta_j),
\end{equation}
However, in introducing $1/\pi$ to ensure that the single-variable sine integrals satisfy \rfw{p_props2}, we have rendered them incapable of satisfying the normalization requirement that would allow us to interpret $\theta$ as a random variable with uniform probability density around the unit circle. Thus the probabilistic interpretation of $p_X$ is useful but not necessary in Theorem \ref{thm:fund_thm} and in its generalizations.

One important aspect of the proof of Theorem \ref{thm:fund_thm} is the introduction of the permutation tensor $\Delta_{\tbf{m}}$, \rfw{perm_tensor2}. With \rfw{fund_thm2} we could generalize this tensor and define it in terms of products over functions rather than variables, and similar to how it was used in the main proof, the tensor could be used to define more general sums, such as those that appear in combinatorial optimization problems, where the state space consists of permutations of elements. 

\section{Acknowledgements}

The author declares no conflicts of interest. 


\end{document}